\documentclass[12pt,a4paper]{amsart}
\usepackage{amsmath, amssymb, amsfonts,enumerate}
\usepackage[all]{xy}
\usepackage{amscd}

\newtheorem{theorem}{Theorem}[section]
\newtheorem{lemma}[theorem]{Lemma}
\newtheorem{corollary}[theorem]{Corollary}
\newtheorem{proposition}[theorem]{Proposition}

 \theoremstyle{definition}
 \newtheorem{definition}[theorem]{Definition}
 \newtheorem{remark}[theorem]{Remark}

 \newtheorem{example}[theorem]{Example}

\newtheorem{question}{Question}

\numberwithin{equation}{section}
\newcommand {\N}{\mathbb{N}} 
\newcommand {\Z}{\mathbb{Z}} 
\newcommand {\Q}{\mathbb{Q}} 

\newcommand{\GG}{\mathcal{G}}

\newcommand{\OO}{\mathcal{O}}
\newcommand{\PP}{\mathcal{P}}

\DeclareMathOperator{\Ker}{Ker}

\DeclareMathOperator{\Stab}{Stab}
\DeclareMathOperator{\Aut}{Aut}

\begin{document}
\title[On the density of periodic configurations]{On the density of periodic configurations in strongly irreducible subshifts}
\date{\today}
\author{Tullio Ceccherini-Silberstein}
\address{Dipartimento di Ingegneria, Universit\`a del Sannio, C.so
Garibaldi 107, 82100 Benevento, Italy}
\email{tceccher@mat.uniroma1.it}
\author{Michel Coornaert}
\address{Institut de Recherche Math\'ematique Avanc\'ee,
UMR 7501,                                             Universit\'e  de Strasbourg et CNRS,
                                                 7 rue Ren\'e-Descartes,
                                               67000 Strasbourg, France}
\email{coornaert@math.unistra.fr}
\date{\today}
\begin{abstract}
Let $G$ be a residually finite group and let $A$ be a finite set.
We prove that if $X \subset A^G$ is a strongly irreducible subshift of finite type 
containing a periodic configuration then periodic configurations are dense in $X$.
The density of periodic configurations  implies in particular that every injective endomorphism of $X$ is surjective and that the group of automorphisms of $X$ is residually finite.
We also introduce a class of subshifts $X \subset A^\Z$, including all strongly irreducible subshifts and all irreducible sofic subshifts, in which periodic configurations are dense. 
    \end{abstract}
\maketitle

\section{Introduction}

Consider a dynamical system $(X,G)$ consisting of a compact Hausdorff topological space $X$ equipped with a continuous action of a group $G$.
A point $x \in X$ is called \emph{periodic} if its orbit  is finite.
 There are many classical examples of ``chaotic" dynamical systems
(e.g.,   the system generated by  Arnold's cat map on the $2$-torus or, more generally, by any Anosov diffeomorphism on the $n$-torus) in which periodic points are dense.
In fact, density of periodic points often appears as one of the axioms in the various definitions of chaos that may be found in the mathematical literature. 
This is the case in Devaney's definition of chaos, which requires density of periodic points together with topological transitivity 
(cf.   \cite{devaney}, \cite{kontorovich}, \cite{cc-chaos}, and the references therein).
When $X$ is metrizable, the density of periodic points in $(X,G)$ has an important ergodic consequence, namely  the existence of a Borel invariant probability measure on $X$ with full support. 
Indeed, if $(\OO_n)_{n \geq 1}$   is a sequence of disjoint finite  orbits whose union  is dense in $X$
 and $(a_n)_{n \geq 1}$ are positive real numbers satisfying $\sum_{n \geq 1} a_n   = 1$,
then 
$$
\mu = \sum_{n \geq 1} \frac{a_n}{\vert \OO_n \vert}\left(\sum_{x \in \OO_n} \delta_x\right)
$$
is an invariant Borel probability measure  whose support is $X$  
(we use  $\delta_x$ to denote the Dirac measure at $x$ and
$\vert \cdot \vert$ to denote cardinality of finite sets). 
\par
 In this   paper, we investigate the question of the density of periodic points for certain symbolic dynamical systems. 
Before stating our main results, let us briefly  recall some basic definitions from symbolic dynamics.
\par
Let $A$ be a finite set, called the \emph{alphabet}, and $G$ a group.
Here we do not make the assumption that $G$ is finitely generated nor even that it is countable.
 The set $A^G := \{x \colon G \to A\}$,
equipped   with its \emph{prodiscrete topology}, i.e., the product topology obtained by taking the discrete topology on each factor $A$ of $A^G$, is  a totally disconnected compact Hausdorff space, 
called the space of \emph{configurations}.
It is metrizable unless $G$ is uncountable and $A$ contains more than one element.
    The group $G$  acts on $A^G$ by the $G$-\emph{shift}, which is defined by the formula
$gx(h) := x(g^{-1}h)$ for all $g,h \in G$ and $x \in A^G$.
A configuration $x \in A^G$ is called \emph{periodic} if it is periodic with respect to the $G$-shift.
\par
A closed $G$-invariant subset of $A^G$ is called a \emph{subshift}.
Let $X \subset A^G$ be a subshift. A map $\tau \colon X \to X$ is called a \emph{cellular automaton} (or an \emph{endomorphism} of $X$)  if $\tau$ is continuous (with respect to the prodiscrete topology)  and $G$-equivariant (i.e., such that $\tau(gx) = g \tau(x)$ for all $(g,x) \in G \times X$).
 One says that the subshift $X$ is \emph{surjunctive} if every injective cellular automaton $\tau \colon X \to X$   is surjective
 (cf. \cite{gottschalk}, \cite{gromov-esav}, \cite{fiorenzi-periodic},
\cite{ca-and-groups-springer}).
The set consisting of all bijective cellular automata $\tau \colon X \to X$ is a group for composition of maps. 
This group is called the \emph{automorphism group} of $X$ and will be denoted by $\Aut(X)$.  
Recall that a group is called \emph{residually finite} if the intersection of its subgroups of finite index is reduced to the identity element.
The class of residually finite groups includes all virtually polycyclic groups (and in particular all finitely generated nilpotent groups and hence all finitely generated abelian groups), all free groups, and all finitely generated linear groups. 
\par
It is a well known fact
(see Proposition  \ref{p:properties-density} below)
that if  $X \subset A^G$ is a subshift containing a dense set of periodic configurations then:
(i) $X$ is surjunctive,
(ii) the group $\Aut(X)$ is residually finite, and 
(iii) if the action of $G$ on $X$ is faithful then  the group $G$ itself is residually finite.
 \par
   One says that a subshift $X \subset A^G$ is \emph{of finite type} if there exist a finite subset 
$\Omega \subset G$ and a subset $\PP \subset A^\Omega$ such 
that $X$ consists of all  $x \in A^G$ for which the restriction of the configuration $g x$ to 
$\Omega$ belongs to $\PP$ for all $g \in G$.
One says that a subshift $X \subset A^G$ is \emph{strongly irreducible} if there exists a finite subset  $\Delta \subset G$  such that,   
if $\Omega_1$ and $\Omega_2$ are finite subsets of $G$ which are sufficiently far apart in the sense  that the sets 
$\Omega_1$ and $\Omega_2\Delta$ do not meet, 
   then, given any two configurations $x_1$ and $x_2$ in $X$, there exists a configuration $x \in X$ which coincides with $x_1$ on $\Omega_1$ and with $x_2$ on $\Omega_2$.  
\par
We shall establish  the following result.

 \begin{theorem}
\label{t:density-periodic}
Let $G$ be a residually finite group  and let $A$ be a finite set.
Suppose that $X \subset A^G$ is a strongly irreducible subshift of finite type
and that there exists a periodic configuration in $X$.
Then $X$   contains a dense set of periodic configurations.
\end{theorem}

As an immediate consequence, we get the following statement.

\begin{corollary}
Let $G$ be a residually finite group  and let $A$ be a finite set.
Suppose that $X \subset A^G$ is a strongly irreducible subshift of finite type containing a periodic configuration.
Then $X$ is surjunctive and its automorphism group $\Aut(X)$ is residually finite.
If in addition $G$ is countable, then $X$ admits an invariant Borel probability measure with full support.
\qed
\end{corollary}

Theorem \ref{t:density-periodic} was previously obtained by Lightwood \cite[Lemma 5.4]{lightwood} in the case   $G = \Z^d$ is a free abelian group of finite rank.
 Lightwood \cite[Lemma 9.2]{lightwood} also
proved that any strongly irreducible subshift of finite type over $\Z^2$ contains a dense set of periodic configurations.  
 According to \cite[Section 9]{lightwood}, \cite[Section 10]{lightwood-2}, and \cite{hochman}, 
 the  question whether every strongly irreducible subshift of finite type over $\Z^d$  
 contains a dense set of periodic configurations 
 remains open for $d \geq 3$.
  Robinson \cite{robinson} (see also \cite[Example 3]{schmidt}) constructed, using Wang tiles,  nonempty subshifts of finite type over $\Z^2$ without periodic configurations.
On the other hand, it is well known that every nonempty subshift of finite type over $\Z$ must contain a periodic configuration and that every irreducible subshift of finite type, and even every irreducible sofic subshift, over $\Z$ contains a dense set of periodic configurations (see   
\cite{fiorenzi-periodic} and Section \ref{sec:strongly-over-Z} below).
It is also known that a subshift of finite type  over $\Z$ is topologically mixing if and only if it is strongly irreducible (see for example \cite[Corollary 6.1]{cc-myhill}).
In fact, when $G = \Z$, strong irreducibility is enough to guarantee density of periodic configurations. 
More precisely, we have the following result whose proof was kindly communicated to us by Benjy Weiss.

\begin{theorem}
\label{t:benjy}
Let $A$ be a finite set and let $X \subset A^\Z$ be a strongly irreducible subshift.
Then  $X$ contains a dense set of periodic configurations.
 \end{theorem}

\begin{corollary}
Let $A$ be a finite set and
let $X \subset A^\Z$ be a strongly irreducible subshift.  
Then $X$ is surjunctive and its automorphism group $\Aut(X)$ is residually finite.
Moreover,  $X$ admits an invariant Borel probability measure with full support.
\qed
\end{corollary}

Let us note that the surjunctivity of strongly irreducible subshifts over $\Z$ also follows  
 from the more general fact that if $G$ is an amenable group then any strongly irreducible subshift over $G$ is surjunctive (see \cite[Corollary 4.4]{cc-myhill}).
\par
Observe that  Theorem \ref{t:density-periodic}  applies in particular to quiescent strongly irreducible subshifts of finite type over residually finite groups (a subshift is called \emph{quiescent} if it contains a constant configuration, i.e., a configuration that is fixed by the shift action).
Examples of such  subshifts are provided by the \emph{golden mean subshift}
which is the subshift over the group $\Z$   consisting of all bi-infinite sequences of $0$'s and $1$'s with no consecutive $1$'s, and by all \emph{generalized golden mean subshifts}
over residually finite groups 
(see Example \ref{ex:golden} for the definition of generalized golden mean subshifts).
In fact, the golden mean subshift and its generalizations are quiescent and have  bounded propagation in the sense of  Gromov \cite{gromov-esav}
(see  Definition \ref{def:bounded-propagation} below)  and it turns out that every subshift with bounded propagation is strongly irreducible and of finite type
(cf. \cite[Proposition 4.10]{fiorenzi-strongly} and Proposition \ref{p:bounded-prop-implies-ft} below).
  
The paper is organized as follows.
In Section \ref{s:background}, we fix notation and present some background material on shifts and subshifts. 
Section \ref{sec:density} contains the proof of Theorem \ref{t:density-periodic}.
The class of W-subshifts, which is a class of subshifts over $\Z$ strictly containing the class of irreducible sofic subshifts and the class of strongly irreducible subshifts, 
is introduced in Section \ref{s:W-subshifts}.
In Section \ref{sec:strongly-over-Z},
we establish the density of periodic configurations in W-subshifts
(Theorem \ref{density-for-W}).
In the final section, we have collected some additional remarks and a list of open questions.

\section{Background material }
\label{s:background}

\subsection{Group actions}
Let $X$ be a set equipped with an action of a group $G$, i.e., a map
$(g,x) \mapsto gx$ from $G \times X$ into $X$ such that
$g_1 (g_2 x) = (g_1 g_2)x$ and $1_G x = x$ for all $g_1,g_2 \in G$ and $x \in X$ (we denote by $1_G$ the identity element of $G$).
The orbit of a point $x \in X$ is the subset
$G x := \{g x : g \in G\}\subset X$ and its stabilizer is the subgroup
$\Stab_G(x) := \{g \in G : g x = x \} \subset G$.
Given a subgroup $H$ of $G$, one says that a point $x \in X$ is \emph{fixed} by $H$  if $H \subset \Stab_G(x)$, i.e., if $hx = x$ for all $h \in H$.
A point $x \in X$ is called \emph{periodic} if its orbit is finite.
The following conditions are equivalent: (1) $x$ is periodic; (2) the subgroup $\Stab_G(x)$ is of finite index in $G$;
(3) there exists a subgroup of finite index $H \subset G$ fixing $x$. 
 \par
Suppose now that $X$ is a topological space. The action of $G$ on $X$ is said to be 
\emph{continuous} if, for each $g \in G$, the map $x \mapsto g x$ is continuous.

 \subsection{Subshifts}
Throughout this paper, $G$ is a group and $A$ is a finite set.
  \par
  If $\Omega \subset G$ is a finite subset, an element $p \in A^\Omega$ is called a \emph{pattern} with \emph{support} $\Omega$.
For $x \in A^G$, we denote by $x\vert_\Omega$ the pattern with support $\Omega$ obtained 
by restricting $x$ to $\Omega$.
Given  subsets $X \subset A^G$ and   $\Omega \subset G$, we define the subset 
$X\vert_\Omega \subset A^\Omega$ by
$X\vert_\Omega = \{x\vert_\Omega : x \in X\}$.
\par
It  follows from the definition of the prodiscrete topology that a subset $X \subset A^G$ is closed in $A^G$ if and only if it satisfies the following condition:
if a configuration $x \in A^G$ is such that  $x\vert_\Omega \in X\vert_\Omega$ for every finite subset $\Omega \subset G$, then $x \in X$.
  \par
 The following fact is well known. We give a proof for completeness.

\begin{proposition}
\label{p:properties-density}
Let $G$ be a group and let $A$ be a finite set.
Suppose that  $X \subset A^G$ is a subshift containing a dense set of periodic configurations. Then
\begin{enumerate}[\rm (i)]
\item
$X$ is surjunctive;
\item
the automorphism group $\Aut(X)$ is residually finite;
\item
if the action of $G$ on $X$ is faithful then  the group $G$ itself is residually finite.
\end{enumerate}
\end{proposition}

\begin{proof}
(i)
Suppose that  $\tau \colon X \to X$ is a  cellular automaton.
Let $x \in X$ be a periodic configuration and $H_x := \Stab_G(x)$.
Let 
$$
F_x := \{y \in X: h y = y \text{ for all } h \in H_x\}
$$ 
denote the set of  configurations in $X$ that are fixed by $H_x$. 
Observe that the set $F_x$ is finite
since $H_x$ is of finite index in $G$ and every element $y \in F_x$ is entirely determined by its restriction to a complete set of representatives of the right cosets of $H_x$ in $G$.
On the other hand, $\tau$ sends $F_x$ into itself since $\tau$ is $G$-equivariant.
If $\tau$ is injective, we deduce that $\tau(F_x) = F_x$.
Since $x \in F_x$, it follows  that $x$ is in the image of $\tau$.
As $\tau(X)$ is closed in $X$ by compactness, the density of periodic points implies that $\tau(X) = X$. This shows that $X$ is surjunctive.
\par 
(ii)
Let $x \in X$ be a periodic configuration.
Keeping the notation introduced in the proof of (i),
we see that each $\tau \in \Aut(X)$ induces by restriction a permutation of the set $F_x$.
This yields a group homomorphism $\rho_x$ from  $\Aut(X)$ into the symmetric group of $F_x$.
As $F_x$ is finite, its kernel  $\Ker(\rho_x)$ is of finite index in $G$.
Since $x \in F_x$ and periodic configurations are dense in $X$,
the intersection of the subgroups $\Ker(\rho_x)$, where $x$ runs over all periodic configurations in $X$, is reduced to the identity map on $X$.
It follows that $\Aut(X)$ is residually finite.
\par
(iii)
If a group $G$ acts faithfully and continuously on
a Hausdorff space  $X$ and periodic points are dense in $X$, then
 $G$ is   residually finite (see for example 
\cite[Theorem 2.7.1]{ca-and-groups-springer}). 
  \end{proof}

\begin{remark}
Assertions (i) and (ii) of Proposition \ref{p:properties-density} are proved in \cite{fiorenzi-periodic} under the hypothesis that $G$ is finitely generated.
\end{remark}
\subsection{Subshifts of finite type and strongly irreducible subshifts} 
Let   $\Omega$ be a finite subset of $G$. 
Given a subset $\PP \subset A^\Omega$,
the subset $X(\Omega,\PP) \subset A^G$ defined by
$$
X(\Omega,\PP) := \{x \in A^G : (g x)\vert_\Omega \in \PP \text{  for all  } g \in G \}
$$
is clearly a subshift of $A^G$. 
One says that a subshift $X \subset A^G$ is \emph{of finite type} if there exist a finite subset $\Omega \subset G$ and a subset $\PP \subset A^\Omega$ such 
that $X = X(\Omega,\PP)$.
The subset  $\Omega$ is then called a \emph{defining window} and that $\PP$ is a \emph{set of defining patterns} for $X$.
\par
One says that a dynamical system $(X,G)$ is a \emph{factor} of a dynamical system $(Y,G)$ if there
exists a surjective continuous $G$-equivariant map $f \colon Y \to X$.
Note that if $(X,G)$ is a factor of $(Y,G)$ and periodic points are dense in $Y$, 
then periodic points are also dense in $X$.
\par
One says that a subshift $X \subset A^G$ is \emph{sofic} if $X$ is a factor of some subshift of finite type, i.e., if there exists a finite set $B$, a subshift of finite type $Y \subset B^G$, and a surjective continuous 
$G$-equivariant map $f \colon Y \to X$. 
\par
Let $\Delta$ and $\Omega$ be finite subsets of $G$. 
 The $\Delta$-\emph{neighborhood} of $\Omega$  is defined as being  the subset of $G$ 
consisting of all elements $g \in G$ such that the set $g \Delta$ meets $\Omega$. 
Thus we have
  \begin{equation*}
  \Omega^{+\Delta} := \{g \in G : g\Delta \cap \Omega \not= \varnothing \} = \Omega\Delta^{-1}.
 \end{equation*}
Note that we have $\Omega^{+\Delta} \subset \Omega^{+ \Delta'}$ if $\Delta \subset \Delta' \subset G$. 
Observe also that we have $\Omega \subset \Omega^{+\Delta}$ whenever $1_G \in \Delta$.

\begin{definition}
\label{def:delta-irred-subshift}
Let $\Delta$ be a finite subset of $G$.
A subshift $X \subset A^G$ is said to be $\Delta$-\emph{irreducible} if it satisfies the following condition:
if $\Omega_1$ and $\Omega_2$ are finite subsets of $G$ such that  
\begin{equation}
\label{e:delta-neighbor-disjoint}
\Omega_1^{+\Delta}  \cap \Omega_2 = \varnothing,
\end{equation}
 then, given any two configurations $x_1$ and $x_2$ in $X$, there exists a configuration $x \in X$ which satisfies
$x\vert_{\Omega_1} = x_1\vert_{\Omega_1}$ and $x\vert_{\Omega_2} = x_2\vert_{\Omega_2}$.
\par
A subshift $X \subset A^G$ is called \emph{strongly irreducible} if there exists a finite subset 
$\Delta$ of $G$  such that $X$ is $\Delta$-irreducible.
 \end{definition}

Note that if a subshift $X \subset A^G$ is $\Delta$-irreducible for some finite subset $\Delta \subset G$, then
$X$ is $\Delta'$-irreducible for any finite subset $\Delta' \subset G$ such that $\Delta \subset \Delta'$.

\begin{remark}
An easy compactness argument (cf.  \cite[Lemma~4.6]{cc-myhill})
shows that we get an equivalent definition of $\Delta$-irreducibility if the sets 
$\Omega_1$ and $\Omega_2$ are allowed to be infinite
in Definition \ref{def:delta-irred-subshift}.
\end{remark}

\begin{remark}
For $G = \Z^d$, strongly irreducible subshifts were introduced in \cite[Definition 1.10]{burton-steif},    in \cite[Section 2]{ward} (under the name \emph{subshifts with strong specification}), and 
in \cite[Definition 2.4]{lightwood} (under the name \emph{square-mixing subshifts}).
They were also introduced in \cite[Definition 4.1]{fiorenzi-strongly} for finitely generated 
groups $G$.
 \end{remark}

  Suppose that a group $G$ acts continuously on a topological space $X$. 
 One says that the action of $G$ on $X$ is \emph{topologically mixing} if, for any pair of nonempty open subsets $U$ and $V$ of $X$, there exists a finite subset $F \subset G$ such that $U  \cap gV 
 \neq \varnothing$ for all $g \in G \setminus F$.
\par 
One says that a subshift $X \subset A^G$ is \emph{topologically mixing}  if the action of $G$ on $X$ is topologically mixing.
This amounts to saying that $X$ satisfies the following condition:
for any finite subset $\Omega \subset G$ and any two configurations  $x_1, x_2 \in X $,
there exists a finite subset $F \subset G$ such that, for all $g \in G \setminus F$, there exists a configuration $x \in X$ such that $x\vert_\Omega = x_1\vert_{\Omega}$ and $ x\vert_{g\Omega} = x_2\vert_{g\Omega}$. 
 \par
 Every strongly irreducible subshift $X \subset A^G$ is topologically mixing  
(see for example \cite[Proposition 3.3]{cc-myhill}).

\begin{remark}
The Ledrappier subshift, i.e.,  the subshift $X \subset \{0,1\}^{\Z^2}$ consisting of all $x\in \{0,1\}^{\Z^2}$ such that $x(m,n) + x(m + 1,n) + x(m ,n+ 1)$ is even for all $(m,n) \in \Z^2$, provides an example of a quiescent subshift of finite type over 
$\Z^2$  which is  topologically mixing but not strongly irreducible 
(cf. \cite[Example 1.8]{burton-steif}).
On the other hand, it is known that every topologically mixing sofic subshift  over $\Z$ is strongly irreducible (see for example \cite[Corollary 1.3]{cc-myhill}). 
\end{remark}

  \begin{definition}
 \label{def:bounded-propagation}
Let $G$ be a group and let $A$ be a finite set.
\par
Let $\Delta$ be a finite subset of $G$.
A subshift $X \subset A^G$ is said to have $\Delta$-\emph{propagation} if it satisfies the following condition:
 if $\Omega$ is a finite subset of $G$  and  $p \in A^\Omega$ is a pattern whose
 restriction to $\Omega \cap g\Delta$ belongs to $X\vert_{\Omega \cap g\Delta}$  for all  $g \in G$,
then $p \in X\vert_\Omega$.
\par
A subshift $X \subset A^G$ is said to have \emph{bounded propagation} if there exists a finite 
subset $\Delta \subset G$ such that $X$ has $\Delta$-propagation.
\end{definition}

Subshifts with bounded propagation were introduced by Gromov \cite[p. 160]{gromov-esav}
 (see also \cite{fiorenzi-strongly}).

\begin{example}[Generalized golden mean subshifts]
\label{ex:golden}
Let $G$ be a group and take $A = \{0,1,\dots,m\}$, where $m$ is a positive integer.
Suppose that we are given a finite sequence $E_1,\dots,E_k$ of finite subsets of $G$.
Consider  the subshift $X \subset A^G$ consisting of all $x \in A^G$ such that, for all $g \in G$ and $1 \leq i \leq k$, there exists an element $h \in gE_i$ satisfying
$x(h)  = 0$.
This is a quiescent subshift since the identically-$0$ configuration is in $X$. 
Moreover,  $X$ has $\Delta$-propagation for $\Delta = E_1 \cup \dots \cup E_k$.
Indeed, suppose that $\Omega \subset G$ is a finite subset and that $p \in A^\Omega$ is an element whose
 restriction to $\Omega \cap g\Delta$ belongs to $X\vert_{\Omega \cap g\Delta}$  for all  $g \in G$.
Then the configuration $x \colon G \to A$ defined by
$$
x(g) =
\begin{cases}
p(g) & \text{ if  }g \in \Omega \\
0 & \text{ if  } g \in G \setminus \Omega
\end{cases}
$$
satisfies $x \in X$ and $x\vert_\Omega = p$.
\par
For $G = \Z$, $m = k = 1$,  and $E_1 = \{0,1\}$, we recover the usual golden mean subshift 
$X \subset \{0,1\}^\Z$  
consisting of all bi-infinite sequences of $0$'s and $1$'s with no consecutive $1$'s.
\par
For $G = \Z^2$, $m = 1$, $k = 2$, $E_1 = \{(0,0),(1,0)\}$ and $E_2 = \{(0,0),(0,1)\}$,
we obtain the  generalized golden mean subshift $X \subset \{0,1\}^{\Z^2}$ also known as the \emph{hard sphere model}
(see \cite[Example 3.2]{lind}).
\par
For $G = \Z^d$, $k = d$, and $E_i = \{0,e_i\}$, $1 \leq i \leq d$, where $e_1,\dots,e_d$ denotes the canonical basis of  $\Z^d$, we get the \emph{generalized hard-sphere model} 
$X \subset \{0,1,\dots,m\}^{\Z^d}$ described in
\cite[Example 1.6]{burton-steif}. 
  \end{example}

The following result was  proved by Fiorenzi \cite[Proposition 4.10]{fiorenzi-strongly} in the   case when $G$ is finitely generated. 

\begin{proposition}
\label{p:bounded-prop-implies-ft}
Let $G$ be a group and let $A$ be a finite set.
Then every subshift  $X \subset A^G$ with bounded propagation
is strongly irreducible and of finite type.
 \end{proposition}

\begin{proof}
Suppose that  $X \subset A^G$ is a subshift with bounded propagation.
Choose a finite subset  $\Delta \subset G$   such that $X$ has $\Delta$-propagation and
consider the set $\PP = X\vert_\Delta \subset A^\Delta$. If $x \in X$ then $(g^{-1}x)\vert_\Delta \in \PP$ for 
all $g \in G$ since $X$ is $G$-invariant.
Thus, we have $X \subset X(\Delta,\PP)$. Conversely, suppose that $x \in X(\Delta,\PP)$.
Let $\Omega$ be a finite subset of $G$.
For each $g \in G$, we have $g^{-1}x \in X(\Delta,\PP)$ since $X(\Delta,\PP)$ is $G$-invariant. Therefore, there exists $y \in X$ such that $(g^{-1}x)\vert_\Delta = y\vert_\Delta$. This implies 
$x\vert_{g\Delta} = (gy)\vert_{g\Delta}$ and hence 
$x\vert_{\Omega \cap g\Delta} = (gy)\vert_{\Omega \cap g\Delta}$. Thus, we have
$x\vert_{\Omega \cap g\Delta} \in X\vert_{\Omega \cap g\Delta}$ for all $g \in G$.
As $X$  has $\Delta$-propagation,
it follows that $x\vert_\Omega \in X\vert_\Omega$.
Since $X$ is closed in $A^G$, we deduce that $x \in X$.
Thus, we have $X = X(\Delta,\PP)$. This shows that $X$ is of finite type.
\par
  After replacing $\Delta$ by $\Delta \cup \{1_G\}$, we can assume that $1_G \in \Delta$. 
Suppose that   $\Omega_1$ and $\Omega_2$ are finite subsets of $G$ such that
\begin{equation}
\label{e:delta-neighbor-disjoint-2}
\Omega_1^{+\Delta} \cap \Omega_2   = \varnothing.
\end{equation}
Note that this implies $\Omega_1 \cap \Omega_2 = \varnothing$ since $\Omega_1 \subset \Omega_1^{+\Delta}$.
Let $x_1,x_2 \in X$.
 Consider the finite subset $\Omega = \Omega_1 \cup \Omega_2 \subset G$
 and the element $p \in A^\Omega$ given by  $p\vert_{\Omega_1} = x_1\vert_{\Omega_1}$ 
 and $p\vert_{\Omega_2} = x_2\vert_{\Omega_2}$ (observe that $p$ is well defined since the sets $\Omega_1$ 
 and $\Omega_2$ are disjoint).
For all $g \in G$, we have $\Omega \cap g\Delta \subset \Omega_1$ 
or $\Omega \cap g\Delta \subset \Omega_2$ by \eqref{e:delta-neighbor-disjoint-2}.
This implies that
$p\vert_{\Omega \cap g\Delta} = x_1\vert_{\Omega_1 \cap g\Delta}$ or
$p\vert_{\Omega \cap g\Delta} = x_2\vert_{\Omega_2 \cap g\Delta}$.
 It follows that $p\vert_{\Omega \cap g\Delta} \in X\vert_{\Omega \cap g\Delta}$ for all $g \in G$.
As $X$ has $\Delta$-propagation,
we deduce that $p \in X\vert_{\Omega}$. Thus, there is an element $x \in X$ such that
  $x\vert_{\Omega_1} = x_1\vert_{\Omega_1}$ and $x\vert_{\Omega_2} = x_2\vert_{\Omega_2}$.
This shows that $X$ is $\Delta$-irreducible and hence strongly irreducible.
 \end{proof}

\begin{remark}
  In \cite[Section 4]{fiorenzi-strongly}, Fiorenzi gave an example of a strongly irreducible subshift of finite type 
$X \subset \{0,1\}^\Z$ without bounded propagation, namely the subshift admitting $\{010,111\}$ as a defining set of forbidden words. Note that $X$ is quiescent since the identically-$0$ configuration is in $X$.
\end{remark}

\begin{remark} 
It is true that any subshift which is a factor of a strongly irreducible subshift is itself strongly irreducible (see \cite[Proposition 3.4]{cc-myhill}).
On the other hand, a subshift which is a factor of a subshift with bounded propagation may fail to have  bounded propagation.
For example, the even subshift $X \subset \{0,1\}^\Z$ (i.e., the subshift formed by  all bi-infinite sequences of $0$s and $1$s in which every chain of $1$s which is bounded by two $0$s has even length) is sofic.
Indeed, $X$
is a factor of the golden mean  subshift $Y \subset \{0,1\}^\Z$. 
A factor map $\tau \colon Y \to X$ is obtained by associating with each sequence 
$y \in Y$ the sequence $\tau(y) \in X$ deduced from $y$ by replacing each occurrence of the word $10$ by the word $11$.
The even subshift does not have bounded propagation since it is not of finite type.
 \end{remark}

\section{Proof of Theorem \ref{t:density-periodic}}
\label{sec:density}

 Choose a periodic configuration $x_0 \in X$ and a finite-index normal subgroup  $K$ of $G$ such that $K \subset \Stab(x_0)$.
\par
  Let $\Delta$ be a finite subset of $G$ such that
$X$ is $\Delta$-irreducible. Up to enlarging $\Delta$ if necessary, we can further assume that $\Delta$ is  a defining window for $X$ and that we have $1_G \in \Delta$.
\par
 Let $x \in X$ and let $\Omega$ be a finite subset of $G$.
Let us show that there exists a periodic configuration $y \in X$ which coincides with $x$ 
on $\Omega$. This will prove the density of periodic configurations in $X$.
\par
Consider the finite subsets   of $G$
 defined  by 
  $$
  \Omega_1 := \Omega^{+\Delta} =\Omega\Delta^{-1} \quad \text{and} \quad  
  \Omega_2 := \Omega_1^{+\Delta^{-1}\Delta} =  \Omega_1\Delta^{-1}\Delta.
  $$
 Observe that we have $\Omega \subset \Omega_1 \subset \Omega_2$ since $1_G \in \Delta$.
 As $X$ is $\Delta$-irreducible, we can find a configuration $z \in X$ 
which  coincides with $x$ on $\Omega$ and with $x_0$ on $\Omega_2 \setminus \Omega_1$.
\par 
Since $G$ is residually finite, we can find a finite-index normal subgroup $L$ of $G$
such that the restriction to $\Omega_2$ of the canonical epimorphism $\rho_L \colon G \to G/L$ is injective.
Consider now the finite-index normal subgroup $H$ of $G$ given by $H := K \cap L$.
Let $\rho_H \colon G \to G/H$ denote the canonical group  epimorphism. 
Define the configuration $y   \in A^G$ by
 $$
y(g) =
\begin{cases}
z(g') &\text{ if  } \rho_H(g)  = \rho_H(g') \text{  for some  } g' \in \Omega_2, \\
x_0(g)  & \text{  otherwise.} 
\end{cases}
$$
First observe that $y$ is well defined since the restriction to $\Omega_2$ of $\rho_H$ is injective.
It is clear that  $h y = y$ for all $h \in H$ so that $y$ is periodic.
 On the other hand, the configurations $y$ and $x$ coincide on $\Omega$, 
since $y$ coincides with $z$  on $\Omega_2$ and $z$ coincides with $x$ on $\Omega \subset \Omega_2$.
Let us prove that $y \in X$.
As $X$ is of finite type with defining window $\Delta$,
it suffices to show that, for each $g \in G$, there exists a configuration $w_g \in X$ such that $y$ coincides with $w_g$ on $g\Delta$.
If the set $\rho_H(g\Delta)$ does not meet $\rho_H(\Omega_1)$, then $y$ coincides with $x_0$  
on $g\Delta$ and we can take $w_g = x_0$.
Suppose now that $\rho_H(g\Delta)$ meets $\rho_H(\Omega_1)$.
This means that there exist $h \in H$ and $\delta_0 \in \Delta$ such that 
$h^{-1}g\delta_0 \in  \Omega_1$. 
Then, for all $\delta \in \Delta$, setting  $k := g\delta$, we have that
$$
h^{-1}k =   (h^{-1}g\delta_0)\delta_0^{-1}\delta \in \Omega_1\Delta^{-1}\Delta = \Omega_2.
$$
 We deduce that
 $$
 y(k) = h y(k) = y(h^{-1}k) = z(h^{-1}k) = hz(k).
 $$
 Therefore, we can take $w_g = h z$ in that case.
 This shows that $y \in X$ and completes the proof that periodic configurations are dense in $X$. 
\qed

\section{W-subshifts}
\label{s:W-subshifts}

In this section we introduce a class of irreducible subshifts over $\Z$, the class of W-subshifts,
and prove that it strictly contains the class of irreducible sofic subshifts and the class of strongly irreducible subshifts.

Let us first introduce some additional notation and recall the definition of the language associated with a subshift over $\Z$.
\par
 Let $A$ be a finite set.
We denote by  $A^*$ the free monoid based on $A$.
Thus, $A^*$ is the set  consisting of all finite words  $w = a_1 a_2 \cdots a_n$, where $n \geq 0$ and $a_i \in A$ for $1 \leq i \leq n$, 
and the monoid operation on $A^*$ is the concatenation  of words.
The \emph{length} of the word $w = a_1 a_2 \cdots a_n$ is the integer $\vert w \vert := n$.
 The identity element in $A^*$ is the empty word.  It is the only word of length $0$.
\par
Given a word $w = a_1 a_2 \cdots a_n \in A^*$  of length $n$, we define a periodic configuration $w^\infty \in A^\Z$   by setting
$w^\infty(i + k n) = a_i$ for all  $1 \leq i \leq n$ and $k \in \Z$.
\par
 One  says that a word $w \in A^*$ appears as a \emph{subword} of a configuration $x \in A^\Z$ if either $w$ is the empty word or there exist $i,j \in \Z$ with $i \leq j$ such that
$w = x(i)x(i+1) \cdots x(j) $.
\par
 Consider now a subshift $X \subset A^\Z$.
The \emph{language} of $X$ is the subset $L(X) \subset A^*$  consisting of all words $w \in A^*$
such that $w$ appears as a subword of some configuration in $X$.
\par
One says that a subshift $X \subset A^\Z$ is \emph{irreducible} if for all words $u,v \in L(X)$ there exists a word $w \in L(X)$ such that $uwv \in L(X)$.
It is clear that every topologically mixing subshift over $\Z$ is irreducible.

\begin{proposition}
\label{p:conditions-W}
Let $A$ be a finite set and let $X \subset A^\Z$ be a subshift.
Let $n_0 \geq 0$ be an integer.
Then the following conditions are equivalent:
\begin{enumerate}[\rm (a)]
 \item
 for all $u,v \in L(X)$
 there exists a word $c \in L(X)$ with length $\vert c \vert \leq n_0$ such that $ucv \in L(X)$; 
\item
the subshift $X$ is irreducible and 
for every $u \in L(X)$ there exists a word $c \in L(X)$ with length $\vert c \vert \leq n_0$ such that $ucu \in L(X)$.
\end{enumerate}
\end{proposition}

\begin{proof}
Condition (a) trivially implies (b).
Conversely, suppose that condition (b) is satisfied and let $u,v \in L(X)$.
By irreducibility, we can find a word $w \in L(X)$ such that $vwu \in L(X)$.
Moreover, there exists a word $c \in L(X)$ with length $\vert c \vert \leq n_0$ such that $vwucvwu \in L(X)$.
This implies $ucv \in L(X)$. Thus, condition (a) is satisfied.
\end{proof}

\begin{definition}
\label{def:weiss-subshift}
Let $A$ be a finite set. 
We say that a subshift $X \subset A^Z$ is a \emph{W-subshift} if   there exists an integer
 $n_0 \geq 0$ satisfying one of the two equivalent conditions of Proposition \ref{p:conditions-W}. 
  \end{definition}

The following statement is well known.

\begin{lemma}
\label{l:property-irr-sofic}
Let $A$ be a finite set and let $X \subset A^\Z$ be an irreducible sofic subshift.
Then there exists an integer $n_0 \geq 0$ satisfying the following property:
for every $u \in L(X)$, there exists a word $c \in L(X)$ with length $\vert c \vert \leq n_0$ such that the periodic configuration $(uc)^\infty$ is in $X$.
\end{lemma}

\begin{proof}
By Lemma 3.3.10 in \cite{lind-marcus},
there exists a strongly connected finite directed graph $\GG$ whose edges are labelled by elements of $A$ such that $X$ is the set of labels of bi-infinite paths in $\GG$.
If $v_1$ and $v_2$ are vertices of $\GG$, denote by $N(v_1,v_2)$ the minimal length of a finite directed path going from  $v_1$ to $v_2$. Then $X$ clearly satisfies the statement by taking $n_0 = \max N(v_1,v_2)$, where $v_1$ and $v_2$ run over all vertices of $\GG$.
\end{proof}

\begin{proposition}
\label{p:irr-sft-W}
Let $A$ be a finite set.
Then every irreducible sofic subshift $X \subset A^\Z$ is a W-subshift.
In particular, every irreducible subshift of finite type $X \subset A^\Z$ is a W-subshift.
\end{proposition}

\begin{proof}
This is an immediate consequence of Lemma \ref{l:property-irr-sofic} since the word $ucu$ appears 
in the configuration $(uc)^\infty$, so that condition (b) of Proposition \ref{p:conditions-W}
is satisfied.
 \end{proof}

 \begin{proposition}
\label{p:si-W}
Let $A$ be a finite set.
Then every strongly irreducible subshift  $X \subset A^\Z$ is a W-subshift.
\end{proposition}

\begin{proof}
Suppose that $X \subset A^\Z$ is strongly irreducible.
 Choose a positive integer   $n_0$ large enough  so that $X$ is $\Delta$-irreducible for 
$\Delta = \{-n_0, - n_0 + 1, \dots, n_0\}$.
Then, for all $u,v  \in L(X)$, we can find $c \in L(X)$ of length $n_0$ such that $ucv \in L(X)$.
Therefore, condition (a) of Proposition \ref{p:conditions-W} is satisfied. 
\end{proof}

\begin{example}
\label{e:not-s-i}
The subshift $X \subset \{0,1\}^\Z$ that is reduced to the two periodic configurations $(01)^\infty$ and $(10)^\infty$ is an irreducible subshift of finite type.
This implies that  $X$ is a W-subshift.
However,  
$X$  is not topologically mixing and therefore not strongly irreducible. 
Thus, the class of strongly irreducible subshifts is strictly contained in the class of W-subshifts. 
 \end{example}

\begin{example}
\label{e:not-sofic}
Consider the subshift $X \subset \{0,1,2\}^\Z$
consisting of all $x \in \{0,1,2\}^\Z$ such that  no words of the form $01^m2^n0$ 
with $m \not= n$
appear in $x$.
Clearly $X$ is strongly irreducible and hence   a W-subshift.
However, it can be shown by means of a pumping argument \cite[Example 3.1.7]{lind-marcus} that $X$ is not sofic. 
It follows that the class of irreducible sofic subshifts over $\Z$ is strictly contained in the class of 
W-subshifts.
 \end{example}

There are W-subshifts which are neither sofic nor strongly irreducible. 
In fact, we have the following result.

\begin{proposition}
\label{p:W-not-is-nor-si}
There exist uncountably many W-subshifts $X \subset \{0,1\}^\Z$ which are neither sofic nor strongly irreducible.
\end{proposition}

\begin{proof}
We first observe that there are only countably many sofic subshifts $X \subset \{0,1\}^\Z$.
Indeed, each sofic subshift $X \subset \{0,1\}^\Z$ is presented by a finite directed graph
with edges labelled by $0$ and $1$ (cf. \cite[Theorem 3.2.1]{lind-marcus}), and there are only countably many such graphs up to isomorphism.
Thus, it suffices to prove the existence of uncountably many W-subshifts $X \subset \{0,1\}^\Z$ which are not strongly irreducible.

Let $\N$ denote the set of nonnegative integers. With every 
$\sigma \in \{2,4\}^\N$ we associate the infinite set of positive odd integers
$$
S(\sigma) = \{1+\sum_{k=0}^{n-1} \sigma(k): n \in \N\}
$$
and the subshift $X_\sigma \subset \{0,1\}^\Z$ consisting of all $x \in \{0,1\}^\Z$
such that no words of the form $01^m0$ with $m \in \N \setminus S(\sigma)$ appear in $x$.
Observe that $X_\sigma = X_{\sigma'}$ if and only if $\sigma = \sigma'$. Thus the set $\{X_\sigma: \sigma \in \{2,4\}^\N\}$ is uncountable.

Let us fix $\sigma \in \{2,4\}^\N$.
Note that for every $m \in \N$ there exists $p_m \in \{0,1,2,3\}$
such that $m+p_m \in S(\sigma)$.
Let us show that $X_\sigma$ satisfies condition (a) in Proposition \ref{p:conditions-W} with
$n_0 = 3$. Let $u,v \in L(X_\sigma)$.
If $u = 1^m$ or $v = 1^n$, with $m,n \in \N$, we can take $c$ to be the empty word. Otherwise, we have 
$u = u'01^m$ and $v = 1^n0v'$ for some words $u',v' \in \{0,1\}^*$ and $m,n \in \N$, and we can take $c = 1^{p_{m+n}}$. This shows that $X_\sigma$ is a W-subshift.
On the other hand, if we take $u=v=0$, we have $u,v \in L(X_\sigma)$, but every $w \in
\{0,1\}^*$ such that $uwv \in L(X_\sigma)$ must have odd length. It follows that
the subshift $X_\sigma$ is not topologically mixing and therefore not strongly
irreducible.
\end{proof}

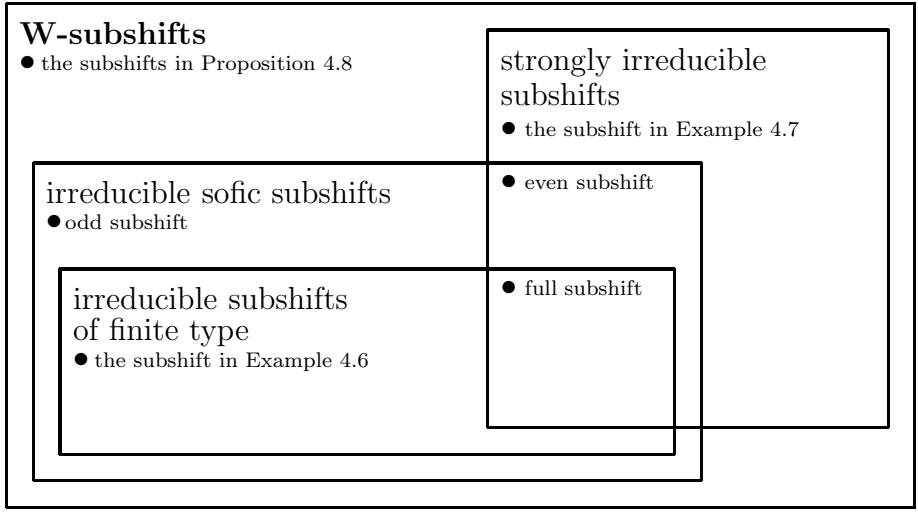
\begin{figure}
\begin{center}
\begin{picture}(360,220)
\thicklines
\put(0,20){\line(1,0){340}}
\put(0,210){\line(1,0){340}}
\put(10,150){\line(1,0){250}}
\put(10,30){\line(1,0){250}}
\put(10,30){\line(0,1){120}}
\put(260,30){\line(0,1){120}}
\put(0,20){\line(0,1){190}}
\put(340,20){\line(0,1){190}}
\put(20,40){\line(1,0){230}}
\put(20,40){\line(0,1){70}}
\put(20,110){\line(1,0){230}}
\put(180,200){\line(1,0){150}}
\put(250,40){\line(0,1){70}}
\put(180,200){\line(0,-1){150}}
\put(180,50){\line(1,0){150}}
\put(330,200){\line(0,-1){150}}
\put(5,195){{\bf W-subshifts}}
\put(15,135){irreducible sofic subshifts}
\put(185,185){strongly irreducible}
\put(185,172){subshifts}
\put(25,95){irreducible subshifts}
\put(25,83){of finite type}
\put(194,140){{\tiny even subshift}}
\put(194,100){{\tiny full subshift}}
\put(185,160){$\bullet$}
\put(194,160){{\tiny the subshift in Example \ref{e:not-sofic}}}
\put(185,140){$\bullet$}
\put(185,100){$\bullet$}
\put(15,125){$\bullet$}
\put(25,73){$\bullet$}
\put(33,73){{\tiny the subshift in Example \ref{e:not-s-i}}}
\put(5,185){$\bullet$}
\put(13,185){{\tiny the subshifts in Proposition \ref{p:W-not-is-nor-si}}}
\put(22,125){{\tiny odd subshift}}
\end{picture}
\end{center}
\caption{Relations among classes of W-subshifts (the \emph{odd subshift} is the
subshift $X \subset \{0,1\}^\Z$ consisting of all configurations $x \in \{0,1\}^\Z$
such that no words of the form $01^n0$ with $n$ even appear in $x$)}
\label{fig:relations}
\end{figure}

\section{Density of periodic configurations in W-subshifts}
\label{sec:strongly-over-Z}

 An immediate consequence of  Lemma \ref{l:property-irr-sofic}
is that every irreducible sofic subshift over $\Z$ contains a dense set of periodic configurations (cf. \cite{fiorenzi-periodic}).
More generally, we have the following result.

\begin{theorem}
\label{density-for-W}
Let $A$ be a finite set and 
let   $X \subset A^\Z$ be a W-subshift.
Then $X$  contains a dense set of periodic configurations.
\end{theorem}

\begin{proof}
 (Benjy Weiss)
It suffices to show that every word  $w \in L(X)$ appears as a subword of some periodic configuration of $X$.
\par
As $X$ is a W-subshift, we can find an integer $n_0 \geq 0$ satisfying condition (b) of Proposition \ref{p:conditions-W}.. 
For each $u\in L(X)$, let $F(u)$
denote the finite set consisting of all words $c \in A^*$ with length $\vert c \vert \leq n_0$ such that  $ucu \in L(X)$.
 Take   $u_0\in L(X)$ such that $F(u_0)$ has minimal cardinality.
Observe that $F(u_0)$ is not empty by our hypothesis on $n_0$. 
 Let us fix  an arbitrary word $c_0 \in F(u_0)$.
\par
Suppose that  $v \in A^*$ is such that  $u_0vu_0\in L(X)$.
 We then have $F(u_0vu_0 ) \subset F(u_0)$ and hence  
\begin{equation*}
\label{e:fillers}
F(u_0 v u_0) = F(u_0)
\end{equation*}
by minimality.
In particular, we have $c_0 \in F(u_0 v u_0)$, that is, $u_0 v u_0 c_0 u_0 v u_0 \in L(X)$.
By induction, it follows  that the sequence of words $(v_k)_{k \geq 1}$ defined by 
$v_1 = u_0 v u_0$ and $v_{k + 1} = v_k c_0 v_k$ for $k \geq 1$, satisfies $v_k \in L(X)$ for all $k \geq 1$.
As
$$
v_{k + 1} = (u_0vu_0c_0)^{2^k - 1}u_0 v u_0
$$
for all $k \geq 1$,
we deduce that
$$
(u_0 v u_0 c_0)^n \in L(X)
$$
for all $n \geq 0$.
As $X$ is  shift-invariant and closed in $A^\Z$, this implies that the periodic configuration $(u_0 v u_0 c_0)^\infty$ is in $X$.
This shows in particular that $v$ appears as a subword of a periodic configuration in $X$.
\par 
Consider now an arbitrary word $w   \in L(X)$. 
Since $X$ is  irreducible, 
we can find 
$u_1 \in A^*$ such that $u_0u_1w \in L(X)$ and $u_2 \in L(X)$ such that $ u_0 u_1 w u_2 u_0  = (u_0 u_1 w) u_2 u_0 \in L(X)$.
Applying our previous argument to $v = u_1 w u_2$, we conclude that $u_1 w u_2$ and hence $w$ appears as a subword of a periodic configuration of $X$. 
This shows  that the periodic configurations are dense in $X$.
 \end{proof}

 \begin{proof}[Proof of Theorem \ref{t:benjy}]
The statement immediately follows from
Proposition \ref{p:si-W} and Theorem \ref{density-for-W}.
 \end{proof}

\section{Concluding remarks and open questions}
\label{s:questions}

 As already mentioned in the introduction,
 the  question whether every strongly irreducible subshift of finite type over $\Z^d$ 
 contains a dense set of periodic configurations 
 remains open for $d \geq 3$. It may be extended as follows.
  
\begin{question}
Let $G$ be a residually finite group. 
 Does every strongly irreducible subshift of finite type over $G$   contain a dense set of periodic configurations?
\end{question}
Note that, by Theorem \ref{t:density-periodic},
it would suffice to prove that such a subshift contains a periodic configuration unless it is  empty.
\par
  Let us observe that  Theorem \ref{t:density-periodic}   becomes false if the hypothesis saying that $X$ is strongly irreducible is replaced by the weaker condition that $X$ is topologically mixing.
Indeed, Weiss \cite[p. 358]{weiss-sgds} described a quiescent topologically mixing subshift of finite type over $\Z^2$ which is not surjunctive.
 \par  
 As mentioned in the introduction,  when $G$ is an amenable group (e.g., $G = \Z^d$),  it is known that every strongly irreducible subshift over $G$ is surjunctive 
(see \cite[Corollary 4.4]{cc-myhill}).
Weiss \cite[p. 358]{weiss-sgds} gave an example of a quiescent subshift of finite type over $\Z$ which is not surjunctive.
It consists of  all bi-infinite sequences of $0$'s,  $1$'s, and $2$'s in which only $00$, $01$, $11$, $12$, and $22$ are allowed to appear among the subwords of length $2$. 
 Gromov \cite{gromov-esav} and Weiss \cite{weiss-sgds} proved that if $G$ is a sofic group then the full shift $A^G$ is surjunctive for any finite alphabet $A$
 (see also \cite[Chapter 7]{ca-and-groups-springer} for an introduction to sofic groups and a presentation of the Gromov-Weiss theorem).
The class of sofic groups is known to be very large.
It includes in particular all residually amenable groups and hence all residually finite group and all amenable groups.  
Actually, no example of a non-sofic group has been found up to now, although
it is widely believed that such examples do exist (cf. \cite[p. 359]{weiss-sgds}). 
  On the other hand, it is unknown whether the full shift $A^G$ is surjunctive for any group $G$ and any finite alphabet $A$ (Gottschalk's conjecture).

\begin{question}
Let $G$ be a residually finite group (resp. a sofic group, resp. a group without any additional assumption).
Is every strongly irreducible subshift over $G$ surjunctive?
\end{question}

In Theorem \ref{t:benjy}, one cannot replace the hypothesis saying  that $X$ is strongly irreducible  by the weaker hypothesis that $X$ is topologically mixing.
Indeed, Petersen \cite{petersen} constructed a topologically mixing minimal subshift over $\Z$ and such a subshift cannot contain periodic configurations by minimality.
On the other hand, minimality clearly implies surjunctivity.
We are therefore  led to the following questions.

\begin{question}
Is every topologically mixing subshift over $\Z$ surjunctive?
\end{question}

\begin{question}
Does every strongly irreducible subshift over $\Z^2$ contain a dense set of periodic configurations ?
\end{question}

Boyle, Lind and Rudolph  \cite{boyle-lind-rudolph} proved that the automorphism group of every subshift of finite type over $\Z$ is residually finite.
They also gave an example   \cite[Example 3.9]{boyle-lind-rudolph}
of a minimal subshift over $\Z$ whose automorphism group is not residually finite (it contains a subgroup isomorphic to $\Q$).
Note that such a subshift contains no periodic configurations by minimality.
On the other hand,  Hochman \cite[Examples 1.4 and 1.5]{hochman-automorphism} gave  examples of a positive entropy subshift of finite type 
and of a topologically  mixing subshift of finite type, both over $\Z^2$,
whose automorphism groups are  not residually finite.
It seems that the following question is open. 

\begin{question}
Is there a strongly irreducible subshift over $\Z^2$ whose automorphism group is not residually finite?
\end{question}

Observe that a positive answer to Question 5 would give a negative answer to Question 4.


\end{document}